\documentclass{amsart}

\usepackage{hyperref}
\begin{document}

\newtheorem{theorem}{Theorem}[section]
\newtheorem{prop}[theorem]{Proposition}
\newtheorem{lemma}[theorem]{Lemma}
\newtheorem{cor}[theorem]{Corollary}
\newtheorem{conj}[theorem]{Conjecture}

\theoremstyle{definition}
\newtheorem{definition}[theorem]{Definition}
\newtheorem{rmk}[theorem]{Remark}
\newtheorem{eg}[theorem]{Example}
\newtheorem{qn}[theorem]{Question}
\newtheorem{defn}[theorem]{Definition}

\numberwithin{equation}{section}

\newcommand\Z{{\mathbb{Z}}}
\newcommand\natls{{\mathbb{N}}}
\newcommand{\C}{{\mathbb C}}
\newcommand{\Rr}{{\mathbb R}}
\newcommand\R{{\mathbb R}}

\newcommand{\bbar}{\overline}
\newcommand{\til}{\widetilde}
\newcommand{\Ra}{\longrightarrow}

\title[Quasiprojective three-manifold groups and complexification]{Quasiprojective
three-manifold groups and complexification of three-manifolds}

\author[I. Biswas]{Indranil Biswas}

\address{School of Mathematics, Tata Institute of Fundamental
Research, Homi Bhabha Road, Bombay 400005, India}

\email{indranil@math.tifr.res.in}

\author[M. Mj]{Mahan Mj}

\address{RKM Vivekananda University, Belur Math, WB 711202, 
India}

\email{mahan.mj@gmail.com, mahan@rkmvu.ac.in}



\date{\today}

\begin{abstract} We characterize the quasiprojective groups that appear as
fundamental groups of compact $3$-manifolds (with or without boundary). We also
characterize all closed $3$-manifolds that admit good complexifications. These answer
questions of Friedl--Suciu, \cite{fs}, and Totaro \cite{tot}.
\end{abstract}

\maketitle

\tableofcontents

\section{Introduction} 

A group is called quasiprojective (respectively, K\"ahler) if it is the
fundamental group of a smooth complex quasiprojective variety (respectively,
compact K\"ahler manifold). 
K\"ahler and quasiprojective $3$-manifold groups have attracted much
attention of late \cite{ds, kotschick, bms, dps, fs, kot13}. In this paper
we characterize quasiprojective $3$-manifold groups. 

We shall follow the convention that our 3-manifolds have {\bf no
spherical boundary components.} Capping such boundary components off by 3-balls does
not change the fundamental group, which is really what interests us here.  

\begin{theorem}[{See Theorem \ref{qpcomb}}]
\label{qpcomb0}
Let $N$ be a compact $3$-manifold (with or without boundary).
If $\pi_1(N)$ is a quasiprojective group, then $N$ is either Seifert-fibered
or $\pi_1(N)$ is one of the following type
\begin{itemize}
\item virtually free, or

\item virtually a surface group.
\end{itemize}
\end{theorem}

Finer results leading to a complete characterization are given in
Section \ref{realzn} and Section \ref{conc} (see Theorem \ref{finalcomb}).
We omit stating these here as they are slightly more complicated to do so.

This characterization of quasiprojective $3$-manifold groups
answers Questions 8.3 and Conjecture 8.4 of \cite{fs}; see Corollary
\ref{fs8.3} and Corollary \ref{fs8.4}.

The following theorem provides an answer to Question 8.1 of \cite{fs}
under mild hypotheses.

\begin{theorem}[{See Theorem \ref{qpfs}}]
Suppose $A$ and $B$ are groups, such that the free product $G\,=\, A*B$ is a 
quasiprojective group. In addition suppose that both $A$ and $B$ admit 
nontrivial finite index subgroups, and at least one of
$A, B$ has a subgroup of index greater than $2$. Then each of $A, B$ are
free products of cyclic groups. In particular both $A$ and $B$ are
quasiprojective groups.
\end{theorem}

A good complexification of a closed smooth manifold $M$ is
defined to be a smooth affine algebraic variety $U$ over the real
numbers such that $M$ is diffeomorphic to
$U (\R)$ (the locus of closed points defined over $\mathbb R$) and
the inclusion $U (\R) \,\Ra\, U (\C)$ is a homotopy equivalence
\cite{tot}. Totaro asks whether a closed smooth manifold $M$ admits a good
complexification if and only if $M$ admits a metric of non-negative
curvature \cite[p. 69, 2nd para]{tot}. As an application of Theorem \ref{qpcomb0},
we prove this in the following strong 
form for $3$-manifolds.

\begin{theorem}[{See Theorem \ref{totaroconj}}]\label{totaroconj0}
A closed $3$-manifold $M$ admits a good complexification if and only if one
of the following hold:
\begin{enumerate}
\item $M$ admits a flat metric,

\item $M$ admits a metric of constant positive curvature,

\item $M$ is covered by the (metric) product of a round $S^2$ and $\Rr$. 
\end{enumerate}
\end{theorem}

Curiously, the proof of Theorem \ref{totaroconj0}
is direct and there is virtually no use of the method or results of
\cite{kul, tot, dps,fs}. Our main tools from recent developments in
$3$-manifolds are: 
\begin{enumerate} 
\item The Geometrization Theorem and its consequences (see \cite{afw}). 
\item Largeness of $3$-manifold groups \cite{agol, wise, ln, clr, lack}.
\end{enumerate}

The basic complex geometric tool is a theorem of Bauer, \cite{bauer},
regarding existence of irrational pencils for quasiprojective 
varieties (the theorem of Bauer is recalled in Theorem \ref{qfibn}).
It is a useful existence result in the same genre as the classical
Castelnuovo-de Franchis Theorem and a theorem of Gromov \cite{gromov-pi1, abckt}.

As a consequence of our results we deduce the 
restrictions on quasiprojective $3$-manifold groups obtained by the authors 
of \cite{dps,fs, kot13} and the restrictions on good complexifications of 
$3$-manifolds deduced in \cite{tot} (this is done in in Section \ref{cons}). We
also indicate, in Remark \ref{closed2},
how to deduce the classification
of (closed) 3-manifold K\"ahler groups \cite{ds,kotschick, bms} using the techniques
of Theorem \ref{qpcomb0},  thus providing a unified treatment of known results.

\section{Preliminaries}

\subsection{Three-manifold groups}

We collect together facts about $3$-manifold groups that will be used here.

By a quasi-K\"ahler manifold we mean the complement of a closed complex analytic
subset of a compact connected K\"ahler manifold.

\begin{defn}
\mbox{}
\begin{enumerate}
\item A group $G$ is {\bf quasiprojective} (respectively, quasi-K\"ahler) if it can
be realized as the fundamental group of a smooth quasiprojective complex variety
(respectively, quasi-K\"ahler manifold).

\item A group $G$ is a {\bf $3$-manifold group} if it can be realized as the
fundamental group of a compact real $3$-manifold (possibly with boundary).

\item A group $G$ is {\bf large} if it has a finite index
subgroup $S$ that admits a surjective homomorphism onto a non-abelian free group. Such
a subgroup $S$ necessarily has a finite index subgroup that admits a surjective homomorphism
onto $F_3$.
\end{enumerate}
\end{defn}

A {\bf prime $3$-manifold} (possibly with boundary) is a $3$-manifold that cannot be decomposed as a non-trivial connected sum.
{\bf Graph manifolds} are prime $3$-manifolds obtained by gluing finitely many 
Seifert-fibered JSJ components along boundary tori. In particular, torus bundles
over a circle are graph manifolds. A 3-manifold $M$ is {\bf geometric} if it is a quotient of one of the following spaces (equipped with standard Riemannian
metrics) by a discrete group
acting freely properly discontinuously via isometries:
 $S^3, {\mathbb E}^3, {\mathbb H}^3,  {\mathbb H}^2 \times {\mathbb R},  {S}^2 \times {\mathbb R}, Nil, Sol, \widetilde{Sl_2({\mathbb R})}$.
In this paper we shall mostly deal with closed 3-manifolds. If $M$ is a compact 3-manifold {\it with} boundary, we say that $M$ is geometric, if the
interior of $M$ is geometric. Note that in this case, the interior of $M$ need not even have finite volume.
 Among the graph manifolds,
$Sol$ and Seifert manifolds are geometric; the rest are non-geometric. It follows
that the gluing maps between the Seifert components in non-geometric manifolds do not
identify circle fibers. (See \cite[p. 59]{afw} and \cite[Ch. 3]{hempel}.)

The following omnibus theorem is the consequence of
the Geometrization theorem of Thurston--Perelman and
work of a large number of people culminating in the resolution of the virtual Haken problem
by Agol and Wise. See \cite{afw} (especially Diagram 1, p. 36) for an excellent account.

\begin{theorem}\label{large}
If a $3$-manifold $M$ has a prime component $N$ satisfying one of the following
three conditions, then the fundamental group of $M$ is large.
\begin{enumerate}
\item $N$ is a compact orientable irreducible $3$-manifold with
non-empty boundary such that M is {\bf not} an $I$-bundle (``$I$'' is a closed interval)
over a surface with non-negative Euler characteristic \cite{clr, lack}.
\item $N$ is closed hyperbolic \cite{agol, wise}.
\item $N$ is a closed, non-geometric graph manifold \cite{ln}.
\end{enumerate}

If $\pi_1(M)$ is a nontrivial free product $G_1 \ast G_2$ (e.g., if $M$ is not
prime), where at least one $G_i$ has order greater than $2$, then
the fundamental group of $M$ is large.
The exceptional case $(\Z/2\Z)\ast(\Z/2\Z)$ is realized only by the connected sum
of two real projective spaces.
\end{theorem}

As an immediate corollary we have the following:

\begin{cor} \label{largecor} 
If the fundamental group of $M$ is not large, then $M$ is Seifert-fibered or
a Sol manifold.
\end{cor} 

A finitely presented group is {\bf coherent} if any finitely generated 
subgroup is finitely presented.

\begin{theorem}[\cite{scott}]
Fundamental groups of compact $3$-manifolds are coherent. \label{coher} \end{theorem}

A consequence is the following \cite[Ch. 11]{hempel}. 

\begin{prop}\label{les} Let $1\,\Ra\, H\,\Ra\, G\,\Ra\, Q\,\Ra\, 1$ be a short
exact sequence of infinite finitely generated groups with $G$ the fundamental group of
a compact orientable $3$-manifold $N$ (possibly with boundary). Then
\begin{enumerate}
\item either $H$ is infinite cyclic and $Q$ is the fundamental group of a compact $2$-orbifold (possibly with boundary), in which case $N$ is Seifert-fibered;

\item or $H$ is the fundamental group of a compact surface (possibly with boundary) and $Q$ is virtually
cyclic. \end{enumerate}
\end{prop}

Another theorem that will be used is:

\begin{theorem}[\cite{bass}] A finitely generated group G is virtually free
if and only if G can be represented as the fundamental group of a
finite graph of groups where all vertex and edge groups are
finite.\label{vfree}
\end{theorem}

\subsection{Logarithmic irrational pencil}

We shall require an extension, due to Bauer, of the classical Castelnuovo-de Franchis theorem
on the existence of an irrational pencil on a projective variety to the more general case of
quasiprojective varieties. We refer to \cite{bauer} for details and quickly recall here the
basic definitions used in this subsection (see also \cite{cat1,di} for related material).
All varieties are defined over $\C$.

A surjective morphism $f\,:\, X \,\longrightarrow\, C$ between quasi-projective varieties is
said to be a {\it fibration} if $f$ has an irreducible (and hence connected) general fiber. If
$C$ is a curve of genus greater than zero, then $f$ is called an {\it irrational pencil}. 

\begin{theorem}[{\cite[p. 442]{bauer}}]\label{qfibn}
Let $X$ be a smooth complex quasiprojective variety such that $\pi_1(X)$ admits a surjective
homomorphism to a group $G$
that admits a finite presentation with $n$ generators and $m$ relations, where
$n-m\,\geq\,3$.
Then there exists an integer $\beta \,\geq\, n-m$ and a quasiprojective curve $C$ with first Betti
number $\beta$ and a logarithmic irrational pencil $f\,:\, X \,\longrightarrow\, C$ with connected fibers.
\end{theorem}

The {\it proof} of Theorem \ref{qfibn} in \cite{bauer} combined with Remark 2.3(1) in \cite{bauer} furnishes the following:

\begin{prop}[\cite{bauer}]\label{qfibnprop}
Let $X$ be a smooth quasiprojective variety, and let $\bbar X$ denote a smooth compactification such that
$\bbar{X} \setminus X$ is a divisor with normal crossings. Further suppose that $\pi_1(X)$ admits a surjection onto a group $G$
that admits a finite presentation with $n$ generators and $m$ relations, where
$n-m\,\geq\,3$. Let $C, f$ be the quasiprojective curve and logarithmic pencil
obtained in Theorem \ref{qfibn}. Let $\bbar C$ denote the projective
completion of $C$. Then there exists $f_1\,:\, \bbar{X}\,\longrightarrow\,
\bbar{C}$ such that $f_1|_X \,=\, f$. In particular, the fibers of $f$
are quasiprojective.
\end{prop}

\begin{proof} Only the last statement (which is really obvious) is not explicitly mentioned in \cite{bauer}. However since we need it explicitly we say a couple of 
words here:

Note that the fibers of $f$ are intersections of fibers of $f_1$ with $X$. All fibers of $f_1$ 
are projective varieties as $f_1$ is algebraic. Hence the fibers of $f$ are quasiprojective.
\end{proof}

 The {\it logarithmic genus} $g^\ast$ of a curve $C$ is defined by the equality $b_1 (C) =
g +g^\ast$, where $g$ is the genus of a smooth completion of $C$. 

Let $X$ be a variety.
A subspace $V\, \subset\, H^1(X,\, {\mathbb C})$ is called \textit{isotropic} if the image
of $\bigwedge^2 V$ in $H^2(X,\, {\mathbb C})$ is zero \cite[p. 441]{bauer}.
A (complex linear) subspace $V\, \subset\, H^1(X,\, {\mathbb C})$ is called 
\textit{real} if $\overline{V}\,=\, V$.

We owe the comment below to the referee:

\begin{rmk} \label{ref} {\rm There is a one-to-one correspondence
between $\Rr-$linear subspaces of $H^1(X,\Rr)$ and real subspaces of $H^1(X,\C)$
in the above sense, that is, $\C-$linear subspaces $V$ such that $\overline{V}\,=\, V$.
The correspondence sends any $\Rr-$linear subspace $W \,\subset\, H^1(X,\,\Rr)$
to $$W \otimes_\Rr \C \,\subset\, H^1(X,\,\Rr)\otimes_\Rr \C\,=\, H^1(X,\,\C)\, .$$
This is the convention we follow.

We could have alternately defined a $\Rr$--linear subspace 
$V$ of $H^1(X,\C)$ to be \textit{real} if $\overline{V}\,=\, V$. 
Now
Theorem \ref{isotropic} below deals with maximal real isotropic subspaces $V$
of $H^1(X,C)$. If $V$ is a real isotropic subspace of $H^1(X,\C)$
in this sense, then $V + \sqrt{-1}V \subset H^1(X,\C)$ is also  isotropic.
Since  $V$ is maximal, it is equal to $V +
\sqrt{-1}V$. So a  maximal real isotropic subspaces $V$
 in this sense is automatically
a complex linear subspace of $H^1(X,\C)$. Thus the two definitions are essentially equivalent.
However ``dimension'' in Theorem \ref{isotropic} below and in
 \cite{bauer} means complex dimension.

The necessary and sufficient condition for $C$ to be complete in Theorem \ref{isotropic} below is slightly
misstated in \cite{bauer}.}
\end{rmk}

It is a standard fact that  the inclusion  $X\subset  \bbar{X}$
induces an injective map from     $H^1(\bbar{X}, \,\C)$ into $H^1(X,\, {\mathbb C})$. We identify $H^1(\bbar{X}, \,\C)$ with its image in
$H^1(X,\, {\mathbb C})$
in the following: 

\begin{theorem}[{\cite[Theorem 2.1]{bauer}, \cite[Theorem 2.11]{cat1}}]\label{isotropic}
Let $X$ be a smooth quasiprojective variety, and let $\bbar X$ denote a
smooth compactification such that $(\bbar{X} \setminus X) = D$ is a divisor with 
normal crossings. Every maximal real isotropic subspace $V\,\subset\,
H^1(X,\, \C)$  
of dimension $\geq 3$ 
 determines a unique 
logarithmic irrational pencil $f\,:\, X\,\Ra\,
C$  onto
a curve $C$ with logarithmic genus $g^\ast \geq 2$. The curve $C$ is complete if and only if $V$
 is a maximal
isotropic real subspace of $H^1(C, {\C})$, and so
$dim_{\C}(V) $ is equal to the genus of $C$.
Else $V = f^\ast (H^1 (C, \, \C))$.
\end{theorem}

We introduce some more notation  towards the final result of this subsection.
For $f\,:\, X\,\longrightarrow\, Y$ be  a fibration of quasiprojective varieties, $Sing (f)
\,\subset\, X$ will denote the set of critical points of $f$.  For any $y\,\in\, Y$, let $F_y
\,:=\,f^{-1}(y)$. Let $F_b$ be a regular fiber of $f$ and $\til{b} \in F_b$.  
  
Proposition \ref{qfibnes} below will use the following Lemma of Nori.

\begin{lemma}[{\cite[Lemma 1.5]{nori}, \cite[Proposition 3.1]{shimada}}]\label{nori}
Let $f: X\rightarrow Y$ be  a fibration of quasiprojective varieties
so that the regular fiber $F_b$  is connected. Let $\iota:F_b \longrightarrow X$ denote
the inclusion map. Let $$\Xi\,\subset\, Y$$ be a Zariski closed subset of codimension
greater than one such that for all $y\in Y\setminus \Xi$ we have $$F_y\setminus(F_y\bigcap
Sing(f))\,\neq \,\emptyset\, .$$
Then $f_\ast\,:\, \pi_1 (X, \til{b})\,\longrightarrow\,\pi_1(Y, b)$ is surjective,
and its  kernel is equal to the image of
$$\iota_*\,:\,\pi_1 (F_b, \til{b})\,\longrightarrow\, \pi_1 (X, \til{b})\, .$$
\end{lemma}

A large part of the proof of the following proposition was supplied by the referee.

\begin{prop}\label{qfibnes}
Let $X, C, f$ be as in Theorem \ref{qfibnprop}. Then there is an exact sequence
$$
1 \,\Ra\, H \,\Ra \,\pi_1(X) \,\Ra \,\pi_1^{orb}(C) \,\Ra \, 1
$$
with $H$ finitely generated, where $\pi_1^{orb}(C)$ denotes the orbifold fundamental group of some orbifold (with finitely many orbifold points)
whose underlying topological space is $C$.
\end{prop}

\begin{proof}
We apply Lemma \ref{nori} with $Y\,=\, C$ and $\Xi \,=\, \emptyset$. 
If $F_y \subset Sing(f)$ for some $y$, then $F_y$ is a multiple fiber.  If every 
singular fiber of $f$ has
an irreducible component with multiplicity one,
 Lemma \ref{nori} directly gives an exact sequence $$1 
\,\Ra\,\iota_\ast (\pi_1 (F_b, \til{b}))\,\Ra\, \pi_1 (X, \til{b}) \,\Ra \,\pi_1(C) 
\,\Ra \, 1\, .$$ Since $\pi_1 (F_b, \til{b})$ is finitely generated by the last 
statement of Proposition \ref{qfibnprop}, so is $H\,=\, \iota_\ast (\pi_1 (F_b, 
\til{b}))\,$ and we are done in this case.

Else suppose there are finitely many points $b_1\, , \cdots\, , b_k$ such that the fibers 
$F_i \,= \,F_{b_i}$ are the multiple fibers of the fibration. Let $Z \,=\, \{ b_1\, , \cdots
\, , b_k \}$ denote the critical set in $C$. 
Suppose that
$F_i = F_{b_i}$ has multiplicity $n_i$, where we define the multiplicity of $F_i$ to be
 the gcd
of the multiplicities of the irreducible components of $F_i$.
We equip $C$ 
with an orbifold structure $C_o$ with orbifold points $b_i$ of order $n_i$.
Since 
$C$ is hyperbolic, so is $C_o$ (since its orbifold Euler characteristic must be 
negative). Hence there exists a finite orbifold-cover $C_1$ of $C_o$ such that $C_1$ 
has no orbifold points \cite{scott1}. This $C_1$ may be thought of as a branched cover of 
$C$ with $n_i$--fold branching at $b_i$. The fibration $f:X \Ra C$ then lifts to a 
fibration $f_1:X_1 \Ra C_1$ where $X_1$ is a manifold cover of $X$ (since $F_i$ is a 
multiple fiber with  multiplicity $n_i$).
Further the multiplicity of each singular fiber of $f_1$ (in the above sense) is one.
It suffices to show therefore that there is an exact sequence $$1 
\,\Ra\,H\,\Ra\, \pi_1 (X_1, \til{b}) \,\stackrel{f_{1\ast}}\Ra \,\pi_1(C_1) \,\Ra \, 1$$ with $H$ 
finitely generated. 

Two things need to be checked now:

\begin{enumerate}
\item $f_{1\ast}$ is surjective.
\item The kernel of $f_{1\ast}$ is the image $H \subset \pi_1(X_1)$
of the fundamental group of a general fiber $F_b$.
\end{enumerate}

Given that $C_1$ has no orbifold points, any (based) loop $\sigma$ can be homotoped slightly to miss the
 singular
set $Z_1$ in $C_1$ without changing its homotopy class.
 Since $f_1$ is a fibration away from the singular set,  $\sigma$ can now be lifted to a (based) loop
$\sigma_1 \subset X_1$ with $f_{1\ast}([\sigma_1])=[\sigma]$ and we conclude that 
$f_{1\ast}$ is surjective.

Let $U = C_1 \setminus Z_1$ and $X_U = f_1^{-1}(U)$. Then $f_1\,:\, X_U\,\longrightarrow
\,U$ is a smooth fibration and hence $\pi_1(X_U)/\pi_1(F_b) = \pi_1(U)$. Next
$\pi_1(X_1)$ is the quotient  of $\pi_1(X_U)$
by the normal subgroup generated by one loop $\sigma_K$ around each
irreducible component $K$ of $X \setminus X_U$. Hence $\pi_1(X_1)/H$ is the quotient of $\pi_1(U)$ by
the normal subgroup generated by $f_{1\ast}([\sigma_K])$. 
If $K \subset F_i$ then $f_{1\ast}([\sigma_K]) = \alpha_i^{n_K}$, where $\alpha_i$ is a small loop around 
the critical point $b_i\in Z_1$ and  $n_K$ is the multiplicity of $K$. If $K_{i1}, \cdots, K_{il}$ are the
irreducible components of $F_i$ with multiplicities $n_{i1}, \cdots, n_{il}$ respectively, then 
$gcd(n_{i1}, \cdots, n_{il})=1$ and hence there exist integers $c_{i1}, \cdots, c_{il}$
such that $\sum_{j=1}^l c_{ij} n_{ij} =1$ and hence $[\alpha_i]$ belongs to the normal subgroup generated by $f_{1\ast}([\sigma_K])$'s.
It follows that the quotient of $\pi_1(U)$ by
the normal subgroup generated by $f_{1\ast}([\sigma_K])$'s is precisely $\pi_1(C_1)$. This
proves the proposition.
\end{proof}

\begin{rmk}
We emphasize that in the proof of Proposition \ref{qfibnes}, we have actually shown the existence of 
a finite manifold cover $X_1$ of $X$ satisfying the exact sequence
$$1 \,\Ra\,H\,\Ra\, \pi_1 (X_1, \til{b}) \,\Ra \,\pi_1(C_1) \,\Ra \, 1$$ with $H$ finitely generated.
\end{rmk}

\section{Quasiprojective three-manifold groups}

In this section, we combine Theorem \ref{large} with Theorem \ref{qfibn} to 
completely characterize quasiprojective $3$-manifold groups.

We shall use the following restriction on quasiprojective groups
due to Arapura and Nori which says that solvable quasiprojective groups
are virtually nilpotent.

\begin{theorem}[\cite{an}]\label{sfas}
Let $N$ be a closed $3$-manifold such that
$\pi_1(N)$ is a quasiprojective group. Then $N$ is not a Sol manifold.
\end{theorem}

\begin{theorem} 
\label{qpmain}
Let $N$ be a closed $3$-manifold, such that
$\pi_1(N)$ is a quasiprojective group. Then $N$ is either Seifert-fibered or
$N$ is finitely covered by $\#_m S^2 \times S^1$.
\end{theorem}

\begin{proof} 
By Theorem \ref{sfas} we can exclude the case where $N$ is a Sol manifold. 
Hence it follows that if $\pi_1(N)$ is not large,
then, by Corollary \ref{largecor}, the manifold $N$ is Seifert-fibered.

Next suppose $\pi_1(N)$ is large. Then there exists a finite index subgroup $G$ of $\pi_1(N)$ such that
$G$ admits a surjection onto the free group $F_3$.

Since $\pi_1(N)$ is quasiprojective, so is $G$. Let $X$ be a smooth
quasiprojective variety with fundamental group $G$. By Theorem \ref{qfibn}, there exists
a logarithmic pencil $f$ (with connected fibers) of $X$ over a
quasiprojective curve $C$ with first Betti number greater than two. 
By passing to a finite sheeted (orbifold) cover of the base if necessary, we can assume without loss of generality that $f$ has no multiple fibers.

By Proposition \ref{qfibnprop}, the generic fiber $F$
is quasiprojective and hence has finitely generated fundamental group. Let $H$ denote
the image of $\pi_1(F)$ in $\pi_1(X)$.
Now we have an exact sequence
$$1 \,\Ra\, H \,\Ra\, \pi_1(X)\,\Ra\, \pi_1(C) \,\Ra\, 1\, .$$ 
If $C$ is closed, it follows from Proposition \ref{les} that $N$ is Seifert fibered. If $H$ is infinite cyclic (or even virtually so), then also  $N$ is 
Seifert fibered.

Else $C$ is quasiprojective non-compact and $H$ is not infinite cyclic. Hence by Proposition \ref{les} again,
the subgroup $H$ is finite and $G$ is virtually free. By Grushko's
theorem, \cite[p. 25, Theorem 3.4]{hempel}, the manifold $N$ is finitely covered
by a connected sum $\#_m S^2 \times S^1$.
\end{proof}

\begin{prop} 
\label{qpmainboundary}
Let $N$ be a $3$-manifold with at least one boundary component of
positive genus. Assume that $\pi_1(N)$ is an infinite quasiprojective group.
Then $\pi_1(N)$ is either virtually free or virtually of the form
$\Z\times F_n$ ($n\geq 1$) or virtually a surface group.
\end{prop}

\begin{proof} By Theorem \ref{large}(1), either $N$ is an $I$--bundle
over a surface of non-negative Euler
characteristic or it is large. If $N$ is an $I$--bundle over a surface of
non-negative Euler characteristic, then $\pi_1(N)$ is either 
$\Z$ or virtually $\Z\oplus \Z$.

Else, by the same argument as in the proof of
Theorem \ref{qpmain}, we have an exact 
sequence $$1\,\Ra\, H\,\Ra\, \pi_1(X)\,\Ra\, \pi_1(C)\,\Ra\,1$$ with $H$ either
$\Z$ or finite, and $C$ a (possibly noncompact) surface. If $H$ is finite,
then $\pi_1(N)$ is either virtually free or virtually a surface group.

If $H$ is $\Z$, then $N$ is Seifert-fibered with base a compact orbifold
surface with boundary. Consequently, $\pi_1(N)$ is either virtually cyclic
or virtually of the form $\Z\times F_n$ with $n\,\geq\, 1$.
\end{proof}

Combining Theorem \ref{qpmain} and Proposition \ref{qpmainboundary} we have
the following:
 
\begin{theorem} 
\label{qpcomb}
Let $N$ be a compact $3$-manifold (with or without boundary) such that
$\pi_1(N)$ is a quasiprojective group. One of the following is true:
\begin{enumerate}
\item $N$ is closed Seifert-fibered,

\item $\pi_1(N)$ is virtually free,

\item $\pi_1(N)$ is virtually of the form $\Z\times
F_n$ with $n\geq 1$,

\item $\pi_1(N)$ is virtually a surface group.
\end{enumerate}
\end{theorem}

\subsection{Refinements and consequences}\label{realzn}

\begin{rmk} \label{finer} The proof of Theorem \ref{qpcomb} gives us a bit 
more. A standing assumption in this section is that $N$ is a compact 
$3$-manifold (with or without boundary) and $\pi_1(N)$ is quasiprojective.

\noindent {\bf Case 1:~} $N$ is closed prime. Then Theorem \ref{qpmain} 
forces $N$ to be Seifert-fibered. \\

\noindent {\bf Case 2:~} $N$ is closed but not prime. Then from Theorem 
\ref{qpmain} the fundamental group $\pi_1(N)$ is virtually free and hence 
by Theorem \ref{vfree}, $\pi_1(N)$ is the fundamental group of a graph of 
groups with edge and vertex groups finite.  Hence
in the prime decomposition of $N$,  each prime component of 
$N$ must have fundamental group that has virtual cohomological dimension  either zero, in which case it is finite; or else 
virtual cohomological dimension  one, in which case it is virtually cyclic. By the classification
of such $3$-manifold groups (see \cite[Theorems 1.1, 1.12]{afw},
\cite[Theorem 9.13]{hempel}), $\pi_1(N)$ 
is of the form $G_1 * G_2 * \cdots * G_k$, where each $G_i$ is either the 
fundamental group of a spherical $3$-manifold or $\Z$ or $\Z\times(\Z/2\Z)$. \\

\noindent {\bf Case 3:~} $N$ is an $I$--bundle over a surface of 
non-negative Euler characteristic. Then $\pi_1(N)$ is either $\Z$ or 
$\Z\oplus \Z$ or the fundamental group of a Klein bottle. It turns out 
(see below) that all these three groups are quasiprojective.\\

\noindent {\bf Case 4:~} $N$ has a boundary component of positive genus and 
$\pi_1(N)$ contains an infinite cyclic normal subgroup. Then by Proposition
\ref{qpmainboundary}, the manifold $N$ is Seifert-fibered with base a compact orbifold 
surface with boundary. In this case a subgroup $G$ of index at most $2$ in 
$\pi_1(N)$ (if $N$ is non-orientable) or one, i.e., $\pi_1(N)$ itself (if 
$N$ is orientable) contains an infinite cyclic central subgroup $\langle t 
\rangle$ such that the quotient $G/\langle t \rangle$ is a free product of 
cyclic groups (finite or infinite) \cite[p. 118]{hempel}.\\

\noindent {\bf Case 5:} $N$ has a boundary component of positive genus and
$\pi_1(N)$ does not contain an infinite cyclic normal subgroup. Then 
by Proposition \ref{qpmainboundary},
\begin{enumerate}
\item either $\pi_1(N)$ is virtually a surface group in which case $N$ is an
$I$--bundle over a surface \cite[Theorem 13.6]{hempel},
\item or after compressing
the boundary as far as possible, $N \,=\, M \# H$, where $H$ is a (possibly
non-orientable) handlebody and hence $\pi_1(H)$ is free, and $M$ is a
closed manifold covered by Case 2.
\end{enumerate}
\end{rmk}

We now demonstrate the converse to Theorem \ref{qpcomb} by describing
examples of smooth quasiprojective varieties that realize the groups occurring in
Remark \ref{finer} as their fundamental groups.
To do this we shall restrict ourselves to {\em orientable}
compact $3$-manifolds with or without boundary.

We start with a lemma that is  well-known to experts. We  provide a proof for completeness (see
\cite[Section 5.3]{bm-kahler} for a closely related construction).

\begin{lemma} Let $X$ be a smooth complex quasiprojective variety, and let $G$ be a
finite group acting by automorphisms on $X$. Then the orbifold
fundamental group of $X/G$ is quasiprojective. \label{qnt}
\end{lemma}

\begin{proof} Let $W$ be a smooth simply connected projective variety admitting a {\em free}
$G$--action by automorphisms. Such varieties exist by a theorem of Serre,
\cite[Example 1.11]{abckt}, which 
says that any finite group is realizable as the fundamental group of a
smooth projective variety.

Let $Y\,=\,X \times W$. Then the diagonal action of $G$ on $Y$ is free and
the (usual) fundamental group of the quotient $Y/G$ coincides with 
the orbifold fundamental group of $X/G$.
\end{proof}

The next proposition addresses Cases (1) and (4) in Remark \ref{finer}.

\begin{prop} \label{sfsqp} Let $N$ be Seifert-fibered with fiber subgroup in
the center of $\pi_1(N)$ such that the base surface
is orientable (with or without boundary). Then $\pi_1(N)$ is quasiprojective.
\end{prop}

\begin{proof}
Let $Q$ be the orientable base orbifold of $N$. Then $Q$ admits the 
structure of an algebraic curve (projective or quasiprojective according 
as $Q$ is without boundary or with boundary). Consider the
quasiprojective orbifold given by $Q$ (after we put a
quasiprojective structure on it). Let $\mathcal L$ be an orbifold algebraic
line bundle on $Q$ such that
\begin{itemize}
\item for each point $x\, \in\, Q$, the action of the isotropy group for
$x$ on the fiber ${\mathcal L}_x$ is faithful, and

\item the degree of $\mathcal L$ is the degree of the Seifert-fibration.
\end{itemize}
Let $L$ denote the underlying variety for the orbifold $\mathcal L$.
Let $\Sigma\, \subset\, L$ be the image of the zero-section of $\mathcal L$.
Then the complement $L \setminus \Sigma$ is a smooth quasiprojective 
variety with the same fundamental group as $N$.
\end{proof}

To address Case (3), we observe first that $\Z$ and $\Z\oplus \Z$ are both
quasiprojective. So only the fundamental group of a Klein bottle
remains. Let $$\phi \,:\,\C^\ast \times \C^\ast \,\Ra\, \C^\ast \times \C^\ast$$
be defined by $(z_1, z_2) \,\longmapsto\,
(\frac{1}{z_1}, -z_2)$. Let $Q\, \subset\, \text{Aut}(\C^\ast \times \C^\ast)$
be the order 2 subgroup generated by $\phi$. Then $Q$ acts freely on
$C$, and the quotient $C/Q$ has the same homotopy type as
a Klein bottle.

In order to completely answer the question ``Which $3$-manifold groups are 
quasi-projective?'', it remains to deal with virtually free groups or virtually 
surface groups. These will be addressed in Section \ref{conc} after developing some 
further tools in Section \ref{gc}.

\subsubsection{Consequences}\label{cons}

We deduce some of the results that preceded this paper from Theorem \ref{qpcomb}.

\begin{theorem}[{\cite[Theorem 1.1]{dps}}]
Let $G$ be the fundamental group of a closed orientable $3$-manifold $M$. Assume
$M$ is formal. Then the following are equivalent.
\begin{enumerate}
\item The Malcev completion of $G$ is isomorphic to the Malcev completion of
a quasi-K\"ahler group.
\item The Malcev completion of $G$
is isomorphic to the Malcev completion of the fundamental group of $S^3$, $\#_n (S^1 \times S^2)$ , or $S^1 \times \Sigma_g$, where $\Sigma_g$ denotes a closed orientable surface of genus $g$ with $g \geq 1$.
\end{enumerate}
\label{dps1}\end{theorem}

\begin{proof} This follows from Theorem \ref{qpmain} by observing that a Seifert-fibered space is formal if and only if it is finitely covered by $S^3$ or a 
trivial circle bundle \cite[Corollary 3.38]{abckt}. 
\end{proof}

\begin{theorem}[{\cite[Theorem 1.2]{fs}}]\label{fsmainthm}
Let $N$ be a $3$-manifold with empty or toroidal boundary.
If $\pi_1(N)$ is a quasiprojective group, then all the closed prime components
of $N$ are graph manifolds.
\end{theorem}

\begin{proof} All the closed prime components of $N$ are in fact Seifert-fibered by Theorem \ref{qpmain} and Remark \ref{finer} Case (5). \end{proof}

\begin{theorem}[\cite{kot13}]\label{kot}
Let $N$ be a $3$-manifold with non-empty boundary.
If $\pi_1(N)$ is a projective group, then
$N$ is an $I$--bundle over a closed orientable surface.
\end{theorem}

\begin{proof} Case 3 and Case 5(1) of Remark \ref{finer} give that $N$ is an
$I$--bundle over a closed surface $S$. If $S$ is non-orientable,
then $\pi_1(S)$ is not projective, hence $\pi_1(N)$ is not projective. 

Case 4 of Remark \ref{finer} forces a finite index subgroup $H$ of $\pi_1(N)$
to be isomorphic to $F_n \times \Z$, with $n>1$. The group $H$ is not projective
and hence $\pi_1(N)$ is not projective.

Case 5 (2) of Remark \ref{finer} along with Theorem \ref{vfree} forces a finite
index subgroup $H$ of $\pi_1(N)$ to be isomorphic to $F_n$, with $n>1$. The
group $H$ is not projective and hence $\pi_1(N)$ is not projective.
\end{proof}

\begin{rmk}\label{closed} Kotschick proves Theorem \ref{kot} in the context of K\"ahler groups. The proof we have given above works equally well in the K\"ahler case.
The only point to be noted is that we have to replace the use of Theorem \ref{qfibn} by the analogous theorem in the K\"ahler context ensuring existence
of irrational pencils as in \cite{gromov-pi1} or \cite{dg}.
\end{rmk}

\begin{rmk} \label{closed2}
 In order to recover the main Theorems of \cite{ds} or \cite{kotschick} from Theorem \ref{qpcomb}
with the modifications mentioned in Remark \ref{closed}, it remains to show that fundamental groups
of circle bundles $N$ over closed surfaces of positive genus are not K\"ahler.
If the bundle is trivial, then $b_1(N)$ is  odd. If the bundle is non-trivial, then the cup
product vanishes identically on $H^1$. Hence the maximal isotropic subspace of $H^1$ has dimension
$2g$, which would imply that $\pi_1(N)$ would admit a surjection onto the fundamental group
of a surface of genus $2g$, a contradiction.
\end{rmk}

Following \cite[p. 69]{tot}, define a good complexification of a closed
manifold $M$ without boundary to be a smooth affine algebraic variety
$U$ over $\mathbb R$ such that $M$ is diffeomorphic to
the space $U(\mathbb R)$ of real points 
and the inclusion $U(\mathbb R)\,\hookrightarrow\, U(\mathbb C)$
is a homotopy equivalence.

Using Theorem \ref{qpmain}, we have an alternative proof of the following
theorem of Totaro.

\begin{theorem}[{\cite[Section 2]{tot}}]
Let $M$ be a closed orientable $3$-manifold with a good complexification.
Then either the cup product $H^1(M,\, {\mathbb Q}) \otimes H^1(M,\,
{\mathbb Q}) \,\Ra\, H^2(M,\, {\mathbb Q})$
is $0$ or $M$ is formal. \label{tot3}
\end{theorem}

\begin{proof}
By Theorem \ref{qpmain}, $M$ is
\begin{enumerate}
\item either finitely covered by $\#_n (S^1 \times S^2)$ in which case the
above cup product is $0$,

\item or $M$ is Seifert-fibered and finitely covered by either $S^3$ or a 
trivial circle bundle over a closed orientable surface; in this case $M$ 
is formal,

\item or $M$ is finitely covered by a non-trivial circle bundle over a
closed surface of positive genus; in this case, the above cup
product is zero.
\end{enumerate}
This completes the proof.
\end{proof}

\begin{rmk}\label{re-stein}
In the definition of a good complexification, if the affine variety over
$\mathbb R$ is weakened to a Stein manifold equipped with an antiholomorphic
involution, then all manifolds admit such a complexification. Indeed, given
a manifold $M$, the total space of the cotangent bundle $T^*M$ admits a
Stein manifold structure \cite{eli-stein, gompf} such that the multiplication by $-1$ on $T^*M$
is an antiholomorphic involution.
\end{rmk}

\section{Classification of three-manifolds with good complexification}\label{gc}

The definition of a good complexification was recalled prior
of Theorem \ref{tot3}. In this Section we shall describe all $3$-manifolds
admitting a good complexification.

\begin{lemma}\label{cover}
If a closed smooth manifold $M$ admits a good complexification, and $M_1$
is a finite-sheeted \'etale cover of $M$, then $M_1$ also admits a good
complexification.
\end{lemma}

\begin{proof}
Let $U$ be a good complexification of $M$. Fix a diffeomorphism of $M$ with
$U(\mathbb R)$. Since the inclusion $U(\mathbb R)
\,\hookrightarrow\, U(\mathbb C)$ induces an isomorphism of fundamental groups,
the covering $M_1$ of $M\,=\, U(\mathbb R)$ has a unique extension to a
covering $U'_1$ of $U(\mathbb C)$. For any point $x\, \in\, U(\mathbb R)$, the
Galois (antiholomorphic) involution
$\sigma$ of $U(\mathbb C)$ for the nontrivial element of
$\text{Gal}({\mathbb C}/\mathbb R)$ induces the identity map of $\pi_1(U(
\mathbb C),\, x)$ because $\sigma\vert_{U(\mathbb R)}\,=\, \text{Id}_{U(
\mathbb R)}$ and the inclusion $U(\mathbb R) \,\hookrightarrow\, U(\mathbb C)$
induces an isomorphism of $\pi_1(U(\mathbb C),\, x)$ with
$\pi_1(U(\mathbb R),\, x)$. Therefore, $\sigma$ has a unique lift
$\sigma'$ to $U'_1$ that fixes $M_1$ pointwise.

The pair $(U'_1 \, , \sigma')$ defines a smooth affine variety
over $\mathbb R$ (see \cite[p. 157, Lemma 4.1]{fs}). Now the variety $(U'_1
\, , \sigma')$ defined over $\mathbb R$ is a good complexification of $M_1$.
\end{proof}

Let $M$ be a closed $3$-manifold admitting a good complexification. From Theorem
\ref{qpmain} it follows that $M$ is either closed Seifert-fibered or is
finitely covered by $\#_m S^2 \times S^1$. We shall therefore consider
separately the following problems:
\begin{enumerate}
\item Which Seifert-fibered manifolds admit good complexifications?
\item Does $\#_m S^2 \times S^1$, ($m>1$), admit a good complexification?
\end{enumerate}

Seifert-fibered $3$-manifolds split into three further
sub-cases according to the orbifold Euler characteristic $\chi (S)$ of the orbifold base $S$ of the
fibration:
\begin{enumerate}
\item[(1a)] $\chi(S)\,> \,0$,

\item[(1b)] $\chi(S)\,=\,0$, and

\item[(1c)] $\chi(S) \,<\, 0$.
\end{enumerate}

First we consider case (1a).
If  $\chi(S)\,> \,0$, then $M$ is covered by $S^3$ or $S^2 \times S^1$
(this follows from the Poincar\'e conjecture and classical $3$-manifold
topology \cite[Theorem 1.12]{afw}). Further, Perelman's solution of the Geometrization conjecture also implies that $M$ is a geometric
quotient of $S^3$ or $S^2 \times S^1$. It is known that geometric quotients of
$S^3$ or $S^2 \times S^1$ admit good complexification \cite[Lemma 3.1]{tot}, \cite{kul}.

\subsection{Seifert-fibered manifolds with base hyperbolic}\label{sfs2}

Now we consider case (1c).

\begin{prop}
Let $N$ be Seifert-fibered with hyperbolic base orbifold. Then $N$ does
not admit a good complexification. \label{sfshyp}
\end{prop}

\begin{proof}
Seifert-fibered manifolds are finitely covered by circle bundles over surfaces.
Since a finite cover of a good complexification
is a good complexification (see Lemma \ref{cover}), it suffices to rule out
principal $S^1$--bundles $N$ over surfaces $S$ with
$\text{genus}(S)\,=\, g \,>\, 1$ and trivial orbifold structure.

So $N$ is now a principal $S^1$--bundle over a compact oriented surface $S$
with $\text{genus}(S)\,=\, g \,>\, 1$.

Let, if possible, $X$ be a good complexification of $N$. Let $X_{\mathbb C}
\,=\, X({\mathbb C})$ be the base change of $X$ to $\mathbb C$.

If the principal $S^1$--bundle $N\, \longrightarrow\, S$ is nontrivial,
then the fundamental group $\pi_1(N)$ admits a presentation
$$
\langle a_1, \cdots, a_g, b_1, \cdots, b_g, t\,\mid\, [a_i,t],  [b_i,t], \prod_{i=1}^g
[a_i,b_i]t^n \rangle\, .
$$
Then $\pi_1(N)$ admits a surjection onto the surface group $\pi_1(\Sigma_g) = 
\langle a_1, \cdots, a_g, b_1, \cdots, b_g\,\mid\,  \prod_{i=1}^g
[a_i,b_i] \rangle\, 
$. Hence, by Theorem \ref{qfibn}, there exists an irrational logarithmic
pencil
\begin{equation}\label{lefpen}
f\,: \,X_{\mathbb C} \,\Ra\, C
\end{equation}
onto a quasiprojective curve $C$ with $b_1(C)\,\geq\, 
(2g-1)\,$. If $C$ is non-compact, then $\pi_1(N)$ must admit a 
surjection onto the free group $F_{2g-1}$, which is impossible as this would induce a surjection of $\pi_1(\Sigma_g)$ (the fundamental group of
the closed orientable surface of genus $g>1$)
onto $F_{2g-1}$. Hence $C$ is 
compact.

Alternatively, if $N$ is the trivial principal $S^1$--bundle over $S$, then
$\pi_1(N)$ admits a surjection onto $\pi_1(S)$. Hence by Theorem \ref{qfibn}, there
exists a logarithmic
pencil as in \eqref{lefpen} onto a quasiprojective
curve $C$ with $b_1(C)\,\geq\, (2g-1)$. If $C$ is is non-compact, then 
$\pi_1(N)$ must admit a surjection onto $F_{2g-1}$ which is impossible. 
Hence $C$ is compact also in this case.

In either case the genus of $C$ is $g$ and $f_\ast\,:\, \pi_1(X({\mathbb C}))\,\longrightarrow \, \pi_1(C)$
has exactly $\langle t\rangle$ as its kernel.

Let
$$
\sigma\,:\, X_{\mathbb C} \,\Ra \, X_{\mathbb C}
$$
denote the antiholomorphic involution corresponding to the nontrivial
element of $\text{Gal}({\mathbb C}/\mathbb R)$. Fix an identification
of $N$ with $X^\sigma_{\mathbb C}\,=\, X({\mathbb R})$. 
The action of $\sigma$ on $H^1(X_{\mathbb C},\, {\mathbb C})$ is trivial
because the inclusion $X^\sigma_{\mathbb C}\, \hookrightarrow\, X_{\mathbb C}$ is a
homotopy equivalence.
There is a natural bijection between the irrational logarithmic pencils
as in \eqref{lefpen} and the maximal real isotropic subspaces
of $H^1(X_{\mathbb C},\, {\mathbb C})$ satisfying
certain conditions (see the first paragraph in \cite[p. 442]{bauer}).
In view of this bijective correspondence, from the fact that
the action of $\sigma$ on $H^1(X_{\mathbb C},\, {\mathbb C})$
is trivial we conclude that the map $f$ in \eqref{lefpen} commutes
with $\sigma$.
In other words, $\sigma$ descends to an antiholomorphic involution
\begin{equation}\label{s1i}
\sigma_1\, :\, C\, \longrightarrow\, C
\end{equation}
of $C$. Note that inclusion
\begin{equation}\label{fpc}
C^{\sigma_1}\, \supset\, f(X^\sigma_{\mathbb C})
\end{equation}
holds, where $C^{\sigma_1}$ is the fixed point set for $\sigma_1$.

Since $f_\ast \,:\, \pi_1(X({\mathbb C}))\,\longrightarrow \,
 \pi_1(C)$ has exactly $\langle t\rangle$ as its kernel, the same is true for 
$(f|_N)_\ast \,:\, \pi_1(N) \,\longrightarrow \, \pi_1(C)$. Since $N$ and $C$ are both Eilenberg--Maclane spaces, it follows that $f$ is homotopic to
the bundle projection map from $N$ (a circle bundle over $C$) to $C$. Hence
the restriction of $f$ to $N\,=\, X^\sigma_{\mathbb C}$ is surjective. Therefore, from \eqref{fpc} it follows
that $C^{\sigma_1}\, =\, C$. This is a contradiction because the identity
map of $C$ is not antiholomorphic. Hence $N$ cannot admit a
good complexification. 
\end{proof}

\subsection{Nil manifolds}

We now consider the second case where the orbifold base of the
Seifert fibration is flat (the genus of the orbifold is $1$).

Non-trivial circle bundles over Euclidean orbifolds are also called 
\textit{nil manifolds}.

\begin{prop}\label{sfsnil}
Let $N$ be a Nil manifold. Then $N$ does not admit a good complexification.
\end{prop}

\begin{proof}
As before, in view of Lemma \ref{cover} it suffices to rule out non-trivial
principal $S^1$--bundles $N$ over the torus with trivial orbifold
structure.

So $N$ is a nontrivial principal $S^1$--bundle over a surface of genus one.

Suppose $X$ is a good complexification
of $N$. As before, let
$$
\sigma\,:\, X_{\mathbb C} \,\Ra \, X_{\mathbb C}
$$
denote the antiholomorphic involution corresponding to the nontrivial
element of $\text{Gal}({\mathbb C}/\mathbb R)$.

Let
\begin{equation}\label{alb-e}
Alb\,:\, X_{\mathbb C} \,\Ra\, C
\end{equation}
be the (quasi) Albanese
map. Then $C$ has fundamental group $\Z\oplus \Z$ and hence $C$ is either
an elliptic curve or the semiabelian variety $\C^\ast \times \C^\ast$
(see \cite{nwy}).

\noindent {\bf Case 1:~}
If $C$ is an elliptic curve, then the same arguments as in
Section \ref{sfs2} now go through as before. It leads to the conclusion
that the real dimension of the fixed set of the
involution $\sigma$ is $4$, which is a contradiction.

\noindent {\bf Case 2:~}
Assume therefore that $C$ is the semiabelian variety $\C^\ast\times\C^\ast$.
If $\dim_\C (Alb(X))\,=\,1$, then $Alb(X)$ is a curve with fundamental
group $\Z\oplus \Z$ and the same argument as in the proof of Proposition \ref{sfshyp} goes through. 

\noindent {\bf Case 3:~}
Hence suppose that $\dim_\C (Alb(X))\,=\,2$, in which case all the fibers of
$Alb$ are quasiprojective curves. 

\noindent {\bf Case 3A:~}
If some fiber of $Alb$ is a singular curve, the same (complex Morse theoretic)
arguments as in \cite[Lemmas 4, 7]{kap} (see also
\cite[Theorem 7.9]{bmp}) show that the kernel
of $Alb_\ast\,:\, \pi_1(X) \,\Ra \,\pi_1 (C)$ is infinitely presented. 

\noindent {\bf Case 3B:~}
Hence the fibers of $Alb$ must all be regular. This forces $\pi_1(F) \,=\, \Z$
and hence $F=\C^\ast$ (since $F$ is a curve). Thus $X$ is a holomorphic
$\C^\ast$--bundle over $\C^\ast \times \C^\ast$.

We note that the involution $\sigma$ commutes with $Alb$. This is because
$Alb$ is the base change to $\mathbb C$ of a morphism between varieties
defined over $\mathbb R$. Therefore, $\sigma$ descends to an antiholomorphic
involution
$$
\sigma_1\,:\, C \,\Ra\, C\, .
$$
Since the fixed point set $C^{\sigma_1}\, \subset\, C$ for the involution 
$\sigma_1$ contains $Alb(X^\sigma_{\mathbb C})$, and
$X^\sigma_{\mathbb C}$ is nonempty, we know that $C^{\sigma_1}$ is nonempty.
Consequently,
$$
C^{\sigma_1}\, =\, S^1\times S^1\, .
$$

Therefore, $X^\sigma_{\mathbb C}\,=\, N$ is a principal $S^1$--bundle
over $C^{\sigma_1}\,=\, S^1\times S^1$. We will show that the first Chern
class of this principal $S^1$--bundle on $C^{\sigma_1}$ vanishes.

The first Chern class of the above principal $S^1$--bundle
over $C^{\sigma_1}$ coincides with the first Chern class of the
principal $\C^\ast$--bundle $X_{\mathbb C}$ in \eqref{alb-e} after we
identify $H^2(\C^\ast \times \C^\ast,\,
{\mathbb Z})$ with $H^2(C^{\sigma_1},\, {\mathbb Z})$ using the
inclusion of $C^{\sigma_1}$ in $C$. Therefore, it suffices to show that
the first Chern class of an algebraic line bundle over
$\C^\ast \times \C^\ast$ vanishes.

Take any algebraic line bundle $L$ over $\C^\ast \times \C^\ast$.
The line bundle $L$ extends to an algebraic line bundle over the
projective surface ${\mathbb P}^1\times {\mathbb P}^1$. To see this,
take the closure in ${\mathbb P}^1\times {\mathbb P}^1$ of any divisor
in $\C^\ast \times \C^\ast$ representing $L$. Let $L'\,\longrightarrow\,
{\mathbb P}^1\times {\mathbb P}^1$ be an extension of $L$. 
Therefore, $c_1(L)\,=\, \iota^* c_1(L')$, where $\iota\, :\,
\C^\ast \times \C^\ast\, \hookrightarrow\, {\mathbb P}^1\times {\mathbb P}^1$
is the inclusion map. But
$$
\iota^*(H^2({\mathbb P}^1\times {\mathbb P}^1,\, {\mathbb Z}))\,=\, 0\, .
$$
Therefore, $c_1(L)\,=\, 0$.

Since $X^\sigma_{\mathbb C}\,=\, N$ is the trivial $S^1$--bundle over
$S^1\times S^1$, we conclude that $N\,=\, S^1\times S^1\times S^1$.
This contradicts the given condition that $N$ is a nil manifold.
\end{proof}

\subsection{Connected sum of copies of $S^2{\times}S^1$} 

Now we consider case (2).

\begin{prop}
Let $N$ be any closed $3$-manifold with virtually free fundamental group and suppose
that $\pi_1(N)$ is not virtually cyclic.
Then $N$ does not admit a good complexification. \label{spheres}
\end{prop}

\begin{proof}
Any closed $3$-manifold with virtually free fundamental group is covered by a
connected sum of copies of $S^2 \times S^1$. Therefore, in view of Lemma \ref{cover},
it is enough to rule out $N\,=\,\#_m S^2 \times S^1$, where $m\,>\,1$.

The argument here follows that in Section \ref{sfs2}. We continue with the same notation. 
By passing to a finite-sheeted cover, we can assume
that $m \,\geq\, 3$. So Theorem \ref{qfibn} applies to give $$f\,:\,
X_{\mathbb C}\,\Ra\, C\, ,$$ where $C$ is a quasiprojective curve with $b_1(C)
\,\geq \,m\,\geq\, 3$. Since $\pi_1(X_{\mathbb C}) = \pi_1(N) = F_m$, this forces $\pi_1(C)$ to equal $F_m$
 and $f_\ast\,:\,\pi_1(X_{\mathbb C}) \,\Ra\, \pi_1(C)$ to be an
isomorphism. Further, $C$ must be noncompact.

As shown in the proof of Proposition \ref{sfshyp}, the morphism $f$
commutes with the antiholomorphic involution $\sigma$ of $X_{\mathbb C}$.
Therefore, $\sigma$ descends to an involution $\sigma_1$ of $C$ (as in
\eqref{s1i}). The fixed point locus $C^{\sigma_1}$ is a disjoint union of
(real) one dimensional proper (embedded) submanifolds of $C$. The image
$f(X^\sigma_{\mathbb C})\, \subset\, C^{\sigma_1}$ is a connected component of 
of $C^{\sigma_1}$, in particular, $f(X^\sigma_{\mathbb C})$ is a connected
proper (embedded) submanifold of $C$ of dimension one.

The inclusion $f(X^\sigma_{\mathbb C})\, \hookrightarrow\, C$ induces an
isomorphism of fundamental groups. On the other hand, we have $b_1(C)
\,\geq \,m\, \geq\, 3$. Therefore, there is no connected
proper (embedded) submanifold of $C$ of dimension one such that
the inclusion induces an isomorphism of fundamental groups.
In view of this contradiction, the proof of the proposition is complete.
\end{proof}

Combining Theorem \ref{qpmain} with Propositions \ref{sfshyp}, \ref{sfsnil} and
\ref{spheres} (along with the Geometrization Theorem) we obtain:

\begin{theorem} If a closed $3$-manifold $M$ admits a good complexification, then
one of the following is true: 
\begin{enumerate}
\item The manifold $M$ admits the structure of a Seifert-fibered space over a spherical
orbifold and is therefore covered by $S^3$ or $S^2 \times S^1$. Hence $M$ either admits a metric of constant positive curvature
or is covered by the (metric) product of a round $S^2$ and $\Rr$. 

\item The manifold $M$ is finitely covered by $S^1 \times S^1 \times S^1$.
Hence $M$ admits a flat metric.
\end{enumerate} \label{totaroconj} \end{theorem}

\section{Virtually free groups and virtually surface groups}\label{conc} 

The \textit{genus} of a complex quasiprojective curve $C$ is defined to be the genus
of its smooth compactification $\overline{C}$.

\begin{lemma}\label{lem1}
Let $X$ be a smooth complex quasiprojective variety and 
$$
f\, :\, X\, \longrightarrow\, C
$$
a nonconstant algebraic map to a quasiprojective complex curve of
positive genus. Let $\iota\, :\, S\, \hookrightarrow\, X$ be a smooth
curve in $X$ such that $f\circ\iota$ is a nonconstant map. Then the dimension
of the image of the pullback homomorphism
$$
\iota^* \,:\, H^1(X,\, {\mathbb R})\, \longrightarrow\,
H^1(S,\, {\mathbb R})
$$
is at least two.
\end{lemma}

\begin{proof}
Let $\overline{X}$ be a smooth compactification of $X$ such that $f$ extends
to a morphism
$$
\overline{f}\, :\, \overline{X}\, \longrightarrow\, \overline{C}
$$
with the image of the extension
$$
\overline{\iota}\, :\, \overline{S}\, \longrightarrow\, \overline{X}
$$
being smooth.

We have $(\overline{f}\circ\overline{\iota})^*(H^0(\overline{C},\,
\Omega_{\overline{C}}))
\,\subset\, \overline{\iota}^*(H^0(\overline{X},\, \Omega_{\overline{X}}))$,
and $(\overline{f}\circ\overline{\iota})^*\, :\, H^0(\overline{C},\,
\Omega_{\overline{C}})\,\longrightarrow\, H^0(\overline{S},\, \Omega_{\overline{S}})$
is injective. Therefore,
$$
\dim\overline{\iota}^*(H^0(\overline{X},\, \Omega_{\overline{X}}))\, \geq\, 1\, .
$$
This implies that
\begin{equation}\label{s1}
\dim_{\mathbb R}\overline{\iota}^*(H^1(\overline{X},\, {\mathbb R}))\,=\,
2 \dim_{\mathbb C} \overline{\iota}^*(H^0(\overline{X},\,
\Omega_{\overline{X}}))\, \geq\,2\, .
\end{equation}
The restriction homomorphism $H^1(\overline{X},\, {\mathbb R})\,\longrightarrow
\, H^1(X,\, {\mathbb R})$ is injective, and $\overline{\iota})\vert_{S}
\,=\,\iota$. Therefore, from \eqref{s1} it follows that
$\dim_{\mathbb R}{\iota}^*(H^1(X,\, {\mathbb R}))\, \geq\, 2$.
\end{proof}

A slight modification of the techniques developed in the proofs of Propositions
\ref{sfshyp} , \ref{sfsnil} and \ref{spheres} yield the following
general result. (This might be regarded as a (weak) ``maps'' version of a theorem
of Catanese \cite[Theorem A']{cat} which provides the analogue for spaces.)

\begin{prop}
Let $X$ be a smooth complex quasiprojective variety and $f\,:\, X\,\Ra\, C$
an irrational logarithmic pencil over a curve $C$ with $b_1(C)\, \geq\, 3$. 
Let $F$ be any regular
fiber of $f$ and $i\,:\, F \, \hookrightarrow\, X$ the inclusion map. Suppose
that the image $i_\ast (\pi_1(F))$ is either infinite cyclic or finite. Let
$A$ be an algebraic automorphism of $X$. Then $A(F)$ is a fiber of $f$. Hence
$A$ induces an algebraic automorphism ${A_0}\,:\,
C \,\Ra\, C$.\label{fibns}
\end{prop}

\begin{proof}
By lifting to a further Galois cover of the base $C$ if necessary, we
can assume that the smooth projective curve $\bbar{C}$ has genus greater than one.

Let $i$ denote the inclusion of $A(F)$ in $X$. Assume that $f\circ i$
is not a constant map. Applying Lemma \ref{lem1} to any smooth curve
$S\, \subset\, A(F)$ such that $f\vert_S$ is not constant, 
we conclude that the dimension of the image of the homomorphism
\begin{equation}\label{is}
i^* \,:\, H^1(X,\, {\mathbb R})\, \longrightarrow\,
H^1(A(F),\, {\mathbb R})
\end{equation}
is at least two. 

Since $A$ is a homeomorphism, from the given condition on $F$ it
follows that $i_\ast (\pi_1(A(F)))\, \subset\, \pi_1(X)$ is either
infinite cyclic or finite. Therefore, the dimension of the image of
the homomorphism
$$
i_* \,:\, H_1(S,\, {\mathbb R}) \, \longrightarrow\, H_1(X,\, {\mathbb R})
$$
is at most one. But this contradicts the observation that the image
of the homomorphism in \eqref{is} is at least two.
Therefore, $f\circ i$ is a constant map.
\end{proof}

The next proposition imposes restrictions on quasiprojective groups that 
are virtually free groups or virtually surface groups.

\begin{prop}
Let $G$ be a quasi-projective group that is virtually a non-abelian free
group or virtually the fundamental group of a closed orientable surface of
genus greater than one. Then there is a short exact sequence of the form
$$
1\,\Ra\, K\,\Ra\, G \,\Ra\, H \,\Ra\, 1\, ,
$$
where $K$ is finite and $H$ is the fundamental group of an orientable orbifold surface (possibly with boundary).
\label{freeqp} \end{prop}

\begin{proof}
Let $X$ be a smooth quasiprojective variety with fundamental group
$G$. Let $X_1$ be a finite Galois \'etale cover of $X$ with fundamental
group $H_1$ such that
\begin{itemize}
\item either $H_1$ is non-abelian free, or

\item $H_1$ is isomorphic to the fundamental group of a closed orientable 
surface of genus greater than one.
\end{itemize}

Let $f\,:\,X_1\,\Ra\, C$ be a logarithmic pencil given by Theorem \ref{qfibn}, and
let $i\, :\, F\,\hookrightarrow\, X_1$ be a regular fiber of $f$.
Then $i_\ast \pi_1(F)$ is finite. The quotient group $Q\,=\,G/H_1$ acts by algebraic
automorphisms on
$X_1$ and hence, by Proposition \ref{fibns}, on $C$ via algebraic automorphisms. Let $K$
be the kernel of the action of $Q$ on $C$. Let $H$ be the orbifold fundamental group of
the quotient $C/Q$. Then we have an exact sequence $$1\,\Ra\,K\,\Ra\,G\,\Ra\,H\,\Ra\,1\, .$$
Also since $Q$ acts on $C$ by holomorphic automorphisms, the quotient $C/Q$ is orientable.
\end{proof}

\begin{prop}
Let $G$ be a quasi-projective $3$-manifold group that is virtually free. Then $G$
is one of the following:
\begin{enumerate} 
\item $G\,=\,\Z$ or $\Z\oplus (\Z/2\Z)$
\item $G\,=\, \ast_i G_i$ where each $G_i$ is cyclic. 
\end{enumerate}
\label{vf3}
\end{prop}

\begin{proof} If $G$ is virtually cyclic, then by the classification of such
$3$-manifold groups (see \cite[Theorems 1.1, 1.12]{afw},
\cite[Theorem 9.13]{hempel}), $G$ is one of $\Z$ or $\Z \oplus (\Z/2\Z)$ or $(\Z/2\Z) \ast(\Z/2\Z)$.

Else $G$ is virtually a non-abelian free group. Let $N$ be a $3$-manifold with
$G\,=\,\pi_1(N)$. Then we are in Case (2) or Case 5(2) of Remark \ref{finer}.
In either case, $G= \ast_i G_i$ where each $G_i$ is either finite or $\Z$ or $\Z\oplus
(\Z/2\Z)$.
By \cite[Theorem 3.11]{scott-wall}, the group $G$ contains no finite normal
subgroup. Hence by Proposition \ref{freeqp}, the group $G$ is isomorphic to the
fundamental group of an orientable orbifold surface $S$. Since $G$ is virtually 
a non-abelian free group, the orbifold surface $S$ must have boundary. The
orbifold fundamental group $G$ of such an $S$ is of the form $G\,= \,
\ast_i G_i$, where each $G_i$ is cyclic. This is because $S$ deformation
retracts onto a wedge $(\vee_i S^1) \bigvee
(\vee_j D_j)$, where each $D_j$ is a quotient of the unit disk by a
finite cyclic group acting with a single fixed point at the origin.
\end{proof}

\begin{prop}
Let $G$ be a quasi-projective $3$-manifold group that is virtually the 
fundamental group of a closed orientable surface of genus greater than 
one. Then $G$ is isomorphic to the fundamental group of a closed 
orientable surface of genus greater than one. \label{vs3}
\end{prop}

\begin{proof}
If $G$ is not isomorphic to the fundamental group of a closed orientable 
surface of genus greater than one, then by Case 5(1) of Remark 
\ref{finer}, the group $G$ contains an index 2 subgroup $H$ that is isomorphic to 
the fundamental group of a closed orientable surface of genus greater than 
one. Also $G$ is isomorphic to the fundamental group of a closed 
non-orientable surface of genus greater than one.

Since such a $G$ contains no finite normal subgroup, by
Proposition \ref{freeqp}, the group $G$ is isomorphic to the fundamental group
of an orientable orbifold surface $S$. No orientable orbifold surface $S$
has the same fundamental group as a closed non-orientable 
surface. Therefore, the proposition follows. 
\end{proof}

Combining the observations in Section \ref{realzn} with those of this 
section, we have the following classification result for quasiprojective 
$3$-manifold groups.

\begin{theorem} \label{finalcomb}
Let $G$ be a quasiprojective group that can be realized as the fundamental
of a compact $3$-manifold $N$ with or without boundary. Then either $N$
is Seifert-fibered, or $G$ satisfies one of the following:
\begin{enumerate} 
\item[(a)] $G$ is isomorphic to $\Z$, $\Z \oplus(\Z/2\Z)$ or the fundamental
group of a Klein bottle or the fundamental group of a closed orientable
surface of genus greater than one.

\item[(b)] $G\,= \,\ast_i G_i$ where each $G_i$ is cyclic. 
\end{enumerate}
Each of the groups appearing in above alternatives (a) and (b) are quasiprojective. If
$N$ is closed Seifert-fibered, and $N$ is spherical, flat or covered by $S^2
\times \Rr$, then $\pi_1(N)$ is quasiprojective. If $N$ is an orientable closed
Seifert-fibered with hyperbolic base orbifold $B$, then $\pi_1(N)$ is
quasiprojective if and only if $B$ is an orientable orbifold.
\end{theorem}

\begin{proof}
All the statements except for the last two are contained in Remark 
\ref{finer}, the examples constructed in Section \ref{realzn} or in 
Proposition \ref{vf3} and Proposition \ref{vs3}. The penultimate statement 
is a consequence of the fact that such manifolds admit good 
complexifications \cite{tot}.

It remains to deal with $N$ an {\em orientable}, Seifert-fibered 
with hyperbolic base orbifold $B$. That an orientable, Seifert-fibered space $N$
with orientable hyperbolic base orbifold $B$ has quasiprojective fundamental
group follows from Proposition \ref{sfsqp} and the last statement in
the first paragraph of \cite[p. 118]{hempel}. We will prove the
converse statement. 

Let $X$ be a smooth quasiprojective variety with $\pi_1(X)\,=\, \pi_1(N)$. Let
$B'$ be an orientable hyperbolic surface (without orbifold points)
that (Galois) covers $B$ and with $b_1(B')\, >\,2$.
There is a corresponding finite (Galois) cover $N'$ of $N$
which is a circle bundle over $B'$. Let $X'$ be the 
Galois \'etale cover of $X$ corresponding to the subgroup $\pi_1(N')$.
By Theorem \ref{qfibn} (or more precisely by Theorem A' of \cite{cat}
which is its generalization to the quasi-K\"ahler context), there is a
pencil $f\,:\, X' \,\Ra\, C$ with $C$ a closed curve (as $N$ is closed).
We are now in the situation of Proposition \ref{fibns}; the deck
transformation group $Q$ induces an algebraic action on $C$
forcing the quotient orientable orbifold $C/Q$ to be orientable.
\end{proof}

The following immediate Corollary of
Theorem \ref{finalcomb} answers Question 8.3 of \cite[p. 166]{fs}.

\begin{cor}\label{fs8.3}
Let $G$ be a quasiprojective group that can be realized as the fundamental
of a closed graph manifold $M$. Then $M$ is Seifert-fibered.
\end{cor}

Friedl and Suciu conjecture the following in \cite{fs}:

\begin{conj}[{\cite[p. 166, Conjecture 8.4]{fs}}]\label{conjc}
Let $N$ be a compact $3$-manifold with empty or toroidal
boundary. If $\pi_1(N)$ is a quasiprojective group and $N$ is not prime,
then $N$ is the connected sum of spherical $3$-manifolds
and manifolds which are either diffeomorphic to
$S^1\times D^2$, $S^1\times S^1\times [0,1]$, or the $3$-torus.
\end{conj}

Following is a strong positive answer to it.

\begin{cor}\label{fs8.4}
Let $N$ be a compact $3$-manifold with empty or toroidal
boundary such that $\pi_1(N)$ is a quasiprojective group and $N$ is not prime. Then 
$N$ is the connected sum of lens spaces, $S^1 \times S^2$ and manifolds which are
diffeomorphic to disk bundles over the circle.
\end{cor}

\begin{proof} We are in Case (b) of Theorem \ref{finalcomb}. Then by the
prime decomposition theorem for $3$-manifolds \cite[Ch. 3]{hempel}, the manifold
$M$ is a connected sum of manifolds with cyclic fundamental group. A complete
list of such manifolds is: lens spaces, $S^1 \times S^2$ and
manifolds which are diffeomorphic to disk bundles over the circle.
\end{proof}

From Theorem \ref{finalcomb} it follows that a closed 
{\bf non-orientable} Seifert-fibered manifold $N$ with hyperbolic base
orbifold such that its orientable double cover $N'$ is
a Seifert-fibered manifold with {\bf non-orientable} hyperbolic base
orbifold cannot have quasiprojective fundamental group, because otherwise
$\pi_1(N')$ is quasiprojective contradicting Theorem \ref{finalcomb}.
The only case that thus remains unanswered by Theorem \ref{finalcomb} is the following: 

\begin{qn} Let $N$ be a closed {\bf non-orientable} Seifert-fibered space with
hyperbolic base orbifold such that its orientable double cover
is a Seifert-fibered space with {\bf orientable} hyperbolic base orbifold. Is
$\pi_1(N)$ quasiprojective? \end{qn}

\subsection{Quasiprojective free products} In \cite{fs}, Friedl and Suciu
ask the following:

\begin{qn}[{\cite[p. 165, Question 8.1]{fs}}]\label{quest:free prod}
Suppose $A$ and $B$ are groups, such that the free product
$A*B$ is a quasiprojective group. Does it follow that $A$ and $B$
are already quasiprojective groups?
\end{qn}

\begin{lemma}\label{qpfsl}
Suppose $A$ and $B$ are  groups, such that the free product $A*B$ is a 
quasiprojective group. In addition suppose that both $A, B$ admit 
nontrivial finite index subgroups and at least one of
$A, B$ has a subgroup of index greater than $2$. Then $A*B$ is
virtually free. 
\end{lemma}

\begin{proof}
Since $A, B$ admit nontrivial finite index subgroups, they also admit 
finite index normal subgroups. By the hypothesis, there exist finite quotients
$A_1$ and $B_1$ (of $A$ and $B$ respectively) of which at least one has order more than
$2$. So $A*B$ admits a surjection onto $A_1 \ast B_1$, and hence a finite index
subgroup $G$ of $A*B$ admits a surjection onto a non-abelian free group with greater than
$2$ generators.

Let $X$ be a smooth quasiprojective variety with fundamental group $G$. By 
Proposition \ref{qfibnes}, there exists an exact 
sequence $$1\,\Ra\, H\,\Ra\,G\,\Ra\, F_n\,\Ra\, 1$$ with $n\,\geq\, 3$
and $H$ finitely generated. Hence $H$ is trivial
\cite[Theorem 3.11]{scott-wall}. It follows that $A*B$ is virtually free.
\end{proof}

Following is a positive answer to Question \ref{quest:free prod}
under mild hypotheses.

\begin{theorem}\label{qpfs}
Suppose $A$ and $B$ are  groups, such that the free product $G\,=\, A*B$ is a 
quasiprojective group. In addition suppose that both $A, B$ admit 
nontrivial finite index subgroups and at least one of
$A, B$ has a subgroup of index greater than $2$. Then each of $A, B$ are
free products of cyclic groups. In particular both $A$ and $B$ are
quasiprojective. 
\end{theorem}

\begin{proof} By Lemma \ref{qpfsl} and Proposition \ref{freeqp},
there is a short exact sequence of the form
$$1 \,\Ra\, K \,\Ra\, G \,\Ra \,H \,\Ra\, 1\, ,$$
where $K$ is finite and $H$ is the
fundamental group of an orientable orbifold surface. The subgroup $K$ is
trivial by \cite[Theorem 3.11]{scott-wall}, and $H$ is virtually free. Hence
as in the proof of Proposition \ref{vf3}, we have $G\,=\, \ast_i G_i$, where
each $G_i$ is cyclic. Therefore, since both $A$ and $B$ are free factors of
$G$, they are free product of cyclic groups. 
Hence $A$ and $B$ are fundamental groups of orientable
orbifold surface. In particular, both $A$ and $B$ are quasiprojective.
\end{proof}

\section*{Acknowledgments}

We thank the referee for detailed helpful comments and especially for Remark \ref{ref}
and a substantial part of the argument in Proposition \ref{qfibnes}.
We thank Stefan Friedl for helpful comments on an earlier draft.
This work began during a visit of the first author to RKM Vivekananda
University. A substantial part of the work was done during a visit of both authors
to Harish-Chandra Research Institute. The final touches were added while the
second author was attending a mini-workshop on K\"ahler Groups organized by
Domingo Toledo and Dieter Kotschick at the Mathematischen Forschungsinstitut,
Oberwolfach. We thank all these institutions for their hospitality.
The first author acknowledges the support of the J. C. Bose Fellowship.


\begin{thebibliography}{ZZZZZZ}

\bibitem[ABCKT]{abckt}
J.~Amor{\'o}s, M.~Burger, K.~Corlette, D.~Kotschick and D.~Toledo, {\it 
Fundamental groups of compact K\"ahler manifolds}. Mathematical Surveys and 
Monographs 44, American Mathematical Society, Providence, RI, 1996.

\bibitem[Ag]{agol} I. ~Agol, The virtual Haken conjecture. With an appendix by I.
Agol, D. Groves, and J. Manning. {\em Doc. Math.} {\bf 18} (2013), 1045--1087

\bibitem[ArNo]{an}
D. Arapura and M. Nori, Solvable fundamental groups of algebraic
varieties and K\"ahler manifolds. {\em Compositio Math.} 
\textbf{116} (1999), 173--188.

\bibitem[AFW]{afw}
M. Aschenbrenner, S. Friedl and H. Wilton, 3-manifold groups. {\em Preprint} 
arXiv:1205.0202v2

\bibitem[Bas]{bass}
H. Bass. Covering theory for graphs of groups. {\em Jour. Pure and Appl.
Alg.} \textbf{89} (1993), 3--47.

\bibitem[Bau]{bauer}
I. ~Bauer, Irrational pencils on non-compact algebraic manifolds.
{\em Internat. J. Math.} \textbf{8} (1997), 441--450.

\bibitem[BiMj]{bm-kahler}
I. Biswas and M. Mj, One relator K\"ahler groups. {\em
Geom. Topol.} {\bf 16} (2012), no. 4, 2171--2186.

\bibitem[BMP]{bmp}
I. Biswas, M. Mj and D. Pancholi,
Homotopical height. {\em Inter. Jour. Math.} (to appear) arXiv:1302.0607.

\bibitem[BMS]{bms}
I. Biswas, M. Mj and H. Seshadri,
Three manifold groups, K\"ahler groups and complex surfaces.
{\em Commun. Contemp. Math.} {\bf 14} (2012), no. 6, 1250038, 24 pp.

\bibitem[Ca]{cat1} F. Catanese, Fibered surfaces, varieties isogeneous to a product 
and related moduli spaces. {\em Amer. J. Math.}  {\bf 122} (2000), 1--44.

\bibitem[Ca]{cat}
F. Catanese, Fibred K\"ahler and quasi-projective groups. {\em Adv. Geom.}
(2003), 13--27. 

\bibitem[CLR]{clr} D. Cooper, D. Long and A. Reid, Essential closed surfaces in bounded
3-manifolds, {\em J. Amer. Math. Soc.} {\bf 10} (1997), 553--563.

\bibitem[DeGr]{dg}
T. Delzant and M. Gromov. Cuts in K\"ahler groups, {\em in: Infinite groups:
geometric, combinatorial and dynamical aspects}, ed. L. Bartholdi,
pp. 31--55, Progress in
Mathematics, Vol. 248, Birkh\"auser Verlag Basel/Switzerland, 2005.

\bibitem[DPS]{dps}
A.~Dimca and S.~Papadima and A.~I. Suciu,
Quasi-K\"ahler groups, 3-manifold groups, and formality. {\em Math.
Zeit.} \textbf{268} (2011), 169--186. 

\bibitem[Di]{di} A. Dimca, On the isotropic subspace theorems. {\em Bull. Math. Soc. 
Sci. Math. Roumanie} \textbf{51(99)} (2008), 307--324.

\bibitem[DiSu]{ds} A. Dimca and A. Suciu, Which $3$-manifold groups are K\"ahler 
groups? {\em J. Eur. Math. Soc.} \textbf{11} (2009), 521--528.

\bibitem[El]{eli-stein}
Y.~Eliashberg.
\newblock Topological characterization of Stein manifolds of dimension $ >2$.
\newblock {\em Internat. J. Math. 1}, pages 29--46, 1990.

\bibitem[FrSu]{fs}
S.~Friedl and A.~I. Suciu,
K\"ahler groups, quasi-projective groups, and 3-manifold groups,
{\it Jour. Lond. Math Soc.} \textbf{89} (2014), 151--168.

\bibitem[Go]{gompf}
R.~E. Gompf.
\newblock A new construction of symplectic manifolds.
\newblock {\em Ann. of Math.} \textbf{142} (1995), 537--696.

\bibitem[Gr]{gromov-pi1} M.~Gromov, Sur le groupe fondamental 
d'une variete k\"ahlerienne. {\em C. R. Acad. Sci. Paris Ser.
I Math.} \textbf{308} (1989), 67--70.

\bibitem[He1]{hempel}
J.\ Hempel, {\em 3-Manifolds}, Ann. of Math. Stud. 86,
Princeton University Press, Princeton, N. J. 1976.

\bibitem[He2]{hempel-rf}
J.\ Hempel,
Residual finiteness for 3-manifolds, Ann. of Math. Stud. 111,
Princeton Univ. Press, Princeton, NJ, 1987  \ vol. {\bf 111} (1987), 379--396.

\bibitem[Ka]{kap} M. Kapovich, On normal subgroups in the fundamental groups of 
complex surfaces. {\em Preprint} arxiv:math/9808085.

\bibitem[Ko1]{kotschick}
D.~Kotschick, Three-manifolds and K\"ahler groups.
{\em Ann. Inst. Fourier} {\bf 62} (2012), 1081--1090.

\bibitem[Ko2]{kot13} D. Kotschick,
K\"ahlerian three-manifold groups. {\em Preprint}, arXiv:1301.1311.

\bibitem[Ku]{kul} R. S. Kulkarni, On complexifications of differentiable
manifolds. {\em Invent. Math.} {\bf 44} (1978), 49--64.

\bibitem[La]{lack} M. ~Lackenby, Some 3-manifolds and 3-orbifolds with large 
fundamental group. {\em Proc. Amer. Math. Soc.} {\bf 135} (2007), 3393--3402.

\bibitem[LoNi]{ln} D. ~Long and G. ~Niblo, Subgroup separability
and 3-manifold groups. {\em Math. Zeit.} {\bf 207}
(1991), 209--215.

\bibitem[NWY]{nwy}
J. Noguchi, J. Winkelmann and K. Yamanoi, The second main theorem for holomorphic
curves into semi-abelian varieties. {\em Acta Math.} {\bf 188} (2002), 129--161.

\bibitem[No]{nori}
M.~V. Nori, Zariski's conjecture and related problems.
 {\em Ann. Sci. \'Ecole Norm. Sup. (4)} {\bf 16} (1983), 305--344.

\bibitem[Sc]{scott}
G. P. Scott, Finitely generated $3$-manifold groups are finitely presented, {\em Jour.
London Math. Soc.} {\bf 6} (1973), 437--440.

\bibitem[Sc]{scott1} G. P. Scott, The geometries of 3-manifolds. {\em Bull. London 
Math. Soc.} {\bf 15} (1983), 401--487.

\bibitem[ScWa]{scott-wall} P. Scott and C. T. C. Wall,
Topological {M}ethods in {G}roup {T}heory,
{\em Homological group theory} (C. T. C. Wall, ed.),
London Math. Soc. Lecture Notes Series, vol. 36, Cambridge Univ. Press, 1979.

\bibitem[Sh]{shimada}
I.~Shimada,  Fundamental groups of algebraic fiber spaces.
{\em Comment. Math. Helv.} {\bf 78} (2003), 335--362.

\bibitem[To]{tot} B.~Totaro, Complexifications of non-negatively curved
manifolds. {\em Jour. Eur. Math. Soc.} {\bf 5} (2003), 69--94. 

\bibitem[Wi]{wise} D. Wise, The structure of groups with a quasi-convex
hierarchy, 189 pages, {\em preprint} (2012). 

\end{thebibliography}
\end{document}